\documentclass[12pt,reqno]{elsarticle}

\usepackage{amsfonts,color,amsmath,amssymb,fancyhdr}
\usepackage{amsthm}
\usepackage{tikz}
\usepackage{graphicx}
\usepackage{subfigure}
\usepackage{color}
\usepackage{url}

\def\Box{\vcenter{\vbox{\hrule\hbox{\vrule
     \vbox to 8.8pt{\hbox to 10pt{}\vfill}\vrule}\hrule}}}

\newcommand{\tr}{\textup{Tr}}

\newcommand{\gu}{{n}}
\newcommand{\q}{{p^e}}
\newcommand{\ra}{\rangle}

\newcommand{\la}{\langle}
\newcommand{\F}{{\mathbb F}}
\newcommand{\Z}{{\mathbb Z}}
\newcommand{\C}{{\mathbb C}}

\newcommand{\cM}{{\mathcal M}}

\newcommand{\PG}{\textup{PG}}

\newcommand{\Cay}{\textup{Cay}}
\newcommand\W{\textup{W}}
\newcommand\Q{\textup{Q}}
\newtheorem{thm}{Theorem}

\newtheorem{lemma}[thm]{Lemma}

\newtheorem{example}[thm]{Example}
\numberwithin{equation}{section}
\numberwithin{thm}{section}

\theoremstyle{definition}

\begin{document}
\newcommand{\stopthm}{\begin{flushright}
\(\box \;\;\;\;\;\;\;\;\;\; \)
\end{flushright}}

\newcommand{\symfont}{\fam \mathfam}

\title{On $m$-ovoids of Symplectic Polar Spaces}
\author[add1]{Tao Feng} \ead{tfeng@zju.edu.cn}
\author[add1]{Ye Wang}\ead{ye$\_$wang@zju.edu.cn}
\author[add2]{Qing Xiang \corref{cor1}}\ead{qxiang@udel.edu}
\cortext[cor1]{Corresponding author}
\address[add1]{School of Mathematical Sciences, Zhejiang University, 38 Zheda Road, Hangzhou 310027, Zhejiang P.R. China}
\address[add2]{Department of Mathematical Sciences, University of Delaware, Newark, DE 19716, USA
}

\begin{abstract}

In this paper, we develop a new method for constructing $m$-ovoids in the symplectic polar space $\W(2r-1,\q)$ from some strongly regular Cayley graphs in \cite{Brouwer1999Journal}. Using this method, we obtain many new $m$-ovoids which can not be derived by field reduction.

 \vspace*{3mm}

\end{abstract}

\maketitle

\section{Introduction}

Let  $e\geq 1, r\geq 2$ be integers, $p$ be a prime, and $\F_\q$ be the finite field of size $\q$. Let $V$ be a $2r$-dimensional vector space over $\F_\q$ and $f$ be a non-degenerate alternating form defined on $V$. The symplectic polar space $\W(2r-1,\q)$ associated with the form $f$ is the geometry consisting of subspaces of $\PG(V)$ induced by the totally isotropic subspaces with respect to $f$. The symplectic polar space $\W(2r-1,\q)$ contains totally isotropic points, lines, planes, etc. Note that since $f$ is alternating, every point of $\PG(V)$ is totally isotropic. Therefore the set of points of $\W(2r-1,\q)$ coincides with the set of points of $\PG(V)$. The (totally isotropic) subspaces of maximum dimension are called {\it maximals} (or {\it generators}) of $\W(2r-1,\q)$. The {\it rank} of $\W(2r-1,\q)$ is the vector space dimension of its maximals, namely $r$.

In this paper, we are concerned with $m$-ovoids in $\W(2r-1,\q)$. An {\it $m$-ovoid} in $\W(2r-1,\q)$  is a set ${\mathcal M}$ of points such that every maximal of $\W(2r-1,\q)$ meets ${\mathcal M}$ in exactly $m$ points. A 1-ovoid in $\W(2r-1,\q)$ is simply called an {\it ovoid}. Ovoids in $\W(2r-1,\q)$ (and more generally in any classical polar space) were first defined by Thas \cite{Thas1981} in 1981. The existence problem for ovoids in $\W(2r-1,\q)$ is completely solved: $\W(3,\q)$ has an ovoid if and only if $p=2$; and $\W(2r-1,\q)$, $r>2$, has no ovoids. The concept of an $m$-ovoid was first defined by Thas \cite{Thas1989Interesting} for generalized quadrangles, and then generalized to that in classical polar spaces by Shult and Thas \cite{Shult1994msystem}. There are some closely related objects, called $i$-tight sets, in $\W(2r-1,\q)$. We will not study $i$-tight sets in this paper, but simply mention that $m$-ovoids and $i$-tight sets of $\W(2r-1,\q)$ can be unified under the umbrella of intriguing sets  \cite{Bamberg2007Tight} of $\W(2r-1,\q)$.

Intriguing sets (in particular, $m$-ovoids) in classical polar spaces have close connections with other geometric and combinatorial structures such as strongly regular graphs and projective two-weight codes, cf. 
\cite{ Bamberg2007Tight, Bamberg2018A, Bamberg2009Tight, Calderbank1986The}. For example, $m$-ovoids in $\W(2r-1,\q)$ turn out to be projective two-intersection sets in $\PG(2r-1,\q)$ and thus give rise to  strongly regular graphs, cf. \cite{Bamberg2007Tight}. There is also a significant relation between projective two-intersection sets and two-weight codes, cf. \cite{Calderbank1986The}. A construction of $m$-ovoids in $\Q^{-}(5,\q)$ via strongly regular Cayley graphs was given in \cite{Bamberg2018A}.

The main problem concerning $m$-ovoids in $\W(2r-1,\q)$ is: For which $m\geq 1$ does there exist an $m$-ovoid in $\W(2r-1,\q)$? As we mentioned above, when $m=1$, this problem is completely solved. In sharp contrast,  the existence problem for $m$-ovoids with $m\geq 2$ is wide open. We give a brief summary of known results here. We start with $\W(3,\q)$: When $p$ is odd,  there are no ovoids in $\W(3,\q)$, cf. \cite{Payne1984Finite}; but there is a partition of $\W(3,\q)$ into $2$-ovoids, so there exists an $m$-ovoid in $\W(3,\q)$ for each even positive integer $m$, cf. \cite{Bamberg2009Tight}; moreover Cossidente  et al. gave a construction of $\frac{(\q+1)}{2}$-ovoids in $\W(3,\q)$ when $p$ is odd in \cite{Cossidente2008On}. When $p=2$, Cossidente et al. gave a construction of $m$-ovoids for all possible $m$ in $\W(3,\q)$ in \cite{Cossidente2008On}. Next we consider the case of $\W(5,\q)$: First there are some sporadic examples of $m$-ovoids in $\W(5,\q)$, cf. \cite{Bamberg2007Tight}; when $p=2$,  Cossidente and Pavese \cite{Cossidente2014Intriguing} gave two constructions of nonclassical $(\q+1)$-ovoids in $\W(5,\q)$ by utilizing relative hemisystems and embedded Suzuki-Tits ovoids of a Hermitian surface. For general $\W(2r-1,\q)$, in terms of necessary conditions, it is proved in \cite{Bamberg2007Tight} that if there exists an $m$-ovoid in $\W(2r-1,\q)$ with $r>2$, then $m\geq \frac{(-3+\sqrt{9+4\q^r})}{2\q-2}$; as for constructions, Cossidente and Pavese \cite{Cossidente2018On} gave a partition of $\W(4n-1,\q)$ into a $\frac{(\q^{(2n-2)}-1)}{\q-1}$-ovoid, a $\q^{(2n-2)}$-ovoid and some $2\q^{(2n-2)}$-ovoidsfor $p^e$ even. Also we should mention an important construction method: by using field reduction, an $m'$-ovoid in low rank classical polar spaces gives rise to an $m$-ovoid in higher rank classical polar space. Specifically with the method of field reduction, an $m'$-ovoid in $W(2r'-1,p^{e'})$ gives rise to an $m$-ovoid in $W(2r-1,p^e)$ if $r'\mid r$ and $re=r'e'$, cf. \cite{Kelly2007Con}. In the end of Section 3 of this paper, we will construct $m$-ovoids in $W(2r-1,p^e)$, with $r$ being a prime, say $p_0$; since $p_0$ has only two factors 1 and $p_0$, an $m$-ovoid in $W(2p_0-1,p^e)$ can not be constructed from an $m'$-ovoid in a symplectic polar space with rank lower than $p_0$ by the field reduction method. 

As can be seen from the above summary there has been very little work on $m$-ovoids in high rank symplectic polar spaces.  In terms of constructions, when $q$ is odd, the only known contruction method for $m$-ovoids in high rank symplectic polar spaces is the field reduction method, cf. \cite{Kelly2007Con}. In this paper, we develope a new construction method which allows us to construct many new $m$-ovoids in high rank symplectic polar spaces. To facilitate the description of our method, we give the following equivalent definition of $m$-ovoids in $\W(2r-1,\q)$.

\begin{lemma}\label{equivdef}

Let $\cM$ be a set of points of $\W(2r-1,\q)$. Then $\cM$ is an $m$-ovoid if and only if
\begin{equation}\label{eqn_def}
|P^{\perp}\cap \cM|=
\begin{cases}
m(\q^{(r-1)}+1)-\q^{(r-1)}, &\text{if }P\in \cM,\\
m(\q^{(r-1)}+1), &\text{otherwise.}	
\end{cases}
\end{equation}

\end{lemma}

For a proof of the lemma, we refer the reader to \cite{Bamberg2007Tight}. The basic idea of our construction of $m$-ovoids is to use a partial converse to Theorem 11 in \cite{Bamberg2007Tight}. Concerning $\W(2r-1,\q)$, Theorem 11 in \cite{Bamberg2007Tight} says that an $m$-ovoid gives rise to a strongly regular Cayley graph over $(\F_\q^{2r},+)$ of negative Latin square type. A partial converse to this statement is true; that is, a  strongly regular Cayley graph over $(\F_\q^{2r},+)$ of negative Latin square type with some special property can give rise to an $m$-ovoid in $\W(2r-1,\q)$ (the special property is the ``self-dual" property; this will be made precise in Theorem~\ref{th main}). To implement this strategy, we start with some strongly regular Cayley graphs ${\rm Cay}(\F_q, D)$ in \cite{Brouwer1999Journal}, and equip the ambient finite field $\F_q$, now viewed as a vector space over a subfield $\F_\q$, with an appropriate non-degenerate alternating form $f$, and show that with respect to $f$, $D$ is ``self-dual", hence the set $\cM$ of projective points obtained from $D$ will satisfy (\ref{eqn_def}), giving rise to an $m$-ovoid in the symplectic polar space $\W(2r-1,p^e)$ of rank $r$. The organization of this paper is as follows. In Section 2, we give some preliminaries on  strongly regular graphs and describe the construction using cyclotomic classes of finite fields in \cite{Brouwer1999Journal}. In Section 3, we first describe our construction strategy, and then give the details of our construction of $m$-ovoids. We conclude the paper with Section 4.

\section{Preliminaries}

A {\it strongly regular graph} srg$(v,k,\lambda,\mu)$ is a simple and undirected graph, neither complete nor edgeless, that has the following properties:

(1) It is a regular graph of order $v$ and valency $k$.

(2) For each pair of adjacent vertices $x,y$, there are exactly $\lambda$ vertices adjacent to both $x$ and $y$.

(3) For each pair of nonadjacent vertices $x,y$, there are exactly $\mu$ vertices adjacent to both $x$ and $y$.\\

For example, the pentagon is an srg$(5,2,0,1)$ and the Petersen graph is an srg$(10,3,0,1)$. The parameters of an srg$(v,k,\lambda,\mu)$ satisfy the following basic relation.

\begin{lemma} \label{rel2}{\em \citep[Section 10.1]{Godsil2001Algebraic}}
Let $\Gamma$ be an srg$(v,k,\lambda,\mu)$. Then 
\begin{equation*}
k(k-\lambda-1)=(v-k-1)\mu.
\end{equation*}	
\end{lemma}

Let $\Gamma$ be a (simple, undirected) graph. The adjacency matrix of $\Gamma$ is the $(0,1)$-matrix $A$ with both rows and columns indexed by the vertex set of $\Gamma$, where $A_{xy} = 1$ when there is an edge between $x$ and $y$ in $\Gamma$ and $A_{xy} = 0$ otherwise. The {\it eigenvalues} of $\Gamma$ are defined to be those of its adjacency matrix $A$.  For convenience we call an eigenvalue of $\Gamma$ {\it restricted} if it has an eigenvector which is not a multiple of the all-ones vector ${\bf 1}$. (For a $k$-regular connected graph, the restricted eigenvalues are simply the eigenvalues different from $k$.)

\begin{thm}\label{char}
For a simple graph $\Gamma$ of order $v$, neither complete nor edgeless, with adjacency matrix $A$, the following are equivalent:
\begin{enumerate}
\item $\Gamma$ is strongly regular with parameters $(v, k, \lambda, \mu)$ for certain integers $k, \lambda, \mu$,
\item $A^2 =(\lambda-\mu)A+(k-\mu) I+\mu J$ for certain real numbers $k,\lambda, \mu$, where $I, J$ are the identity matrix and the all-ones matrix, respectively, 
\item $A$ has precisely two distinct restricted eigenvalues $\alpha_1,\alpha_2$.
\end{enumerate}
\end{thm}

For a proof of Theorem~\ref{char}, we refer the reader to \cite{Brouwer2012Spectra}. For later use, we write down the explicit relations between the paramenters of an srg$(v,k,\lambda,\mu)$ and its restricted eigenvalues $\alpha_1,\alpha_2$.

\begin{lemma}\label{relation}{\em \citep[Section 10.2]{Godsil2001Algebraic}}
Let $\Gamma$ be an srg$(v,k,\lambda,\mu)$ with restricted eigenvalues $\alpha_1$, $\alpha_2$, where $\alpha_1>\alpha_2$. Then
\begin{equation}\label{eigenvalue}
\begin{split}
&\alpha_1=\frac{(\lambda-\mu)+\sqrt{(\lambda-\mu)^2+4(k-\mu)}}{2},\\
&\alpha_2=\frac{(\lambda-\mu)-\sqrt{(\lambda-\mu)^2+4(k-\mu)}}{2},
\end{split}	
\end{equation}
with multiplicities $m_1=\frac{1}{2}((v-1)-\frac{2k+(v-1)(\lambda-\mu)}{\sqrt{(\lambda-\mu)^2+4(k-\mu)}})$ and $m_2= \frac{1}{2}((v-1)+\frac{2k+(v-1)(\lambda-\mu)}{\sqrt{(\lambda-\mu)^2+4(k-\mu)}}) $ respectively. 

\end{lemma}

An effective method to construct strongly regular graphs is by the Cayley graph construction. Let $G$ be an additively written abelian group of order $v$, and let $D$ be a subset of $G$ such that $0\not\in D$ and $-D=D$, where $-D=\{-d\mid d\in D\}$. The {\it Cayley graph on $G$ with connection set $D$}, denoted by ${\rm Cay}(G,D)$, is the graph with the elements of $G$ as vertices; two vertices are adjacent if and only if their difference belongs to $D$. Let $\hat{G}$ be the (complex) character group of G. All the eigenvalues of $\Cay(G,D)$ are given by $\psi(D):=\sum_{d\in D}\psi(d)$, $\psi\in \hat{G}$.  Note that $\psi_0(D)=|D|$, where $\psi_0$ is the principal character of $G$.  By Theorem~\ref{char}, the graph $\Cay(G,D)$ is strongly regular if and only if $D$ generates $G$ and $\{\psi(D) : \psi \in \hat{G}\setminus \{ \psi_0 \}\}=\{\alpha_1,\alpha_2\}$ with $\alpha_1\neq \alpha_2$. If this is the case, then the connection set $D$ is called a {\em partial difference set}, and the Delsarte \emph{dual} of $D$ is defined to be either of $\{\psi \in \hat{G}:\,\psi(D)=\alpha_i\}$, $i=1,\,2$.

An srg$(v,k,\lambda,\mu)$ is said to be of {\it Latin square type} (resp. {\it negative Latin square type}) if $(v,k,\lambda,\mu)=(\gu^2,a(\gu-\epsilon),\epsilon \gu+a^2-3\epsilon a,a^2-\epsilon a)$ and $\epsilon=1$ (resp. $\epsilon=-1$). 

Let $G$ and $D$ be the same as above. Suppose that ${\rm Cay}(G,D)$ is an srg$(v,k,\lambda,\mu)$. Then, by Lemma \ref{relation}, one of the duals of $D$ has the same size as that of $D$ if and only if 
$$k= \frac{1}{2}((v-1)-\frac{2k+(v-1)(\lambda-\mu)}{\sqrt{(\lambda-\mu)^2+4(k-\mu)}})\; {\rm or}\; k= \frac{1}{2}((v-1)+\frac{2k+(v-1)(\lambda-\mu)}{\sqrt{(\lambda-\mu)^2+4(k-\mu)}}),$$
which in turn is equivalent to
\begin{equation}\label{eq size}
\frac{(2k+(v-1)(\lambda-\mu))^2}{(\lambda-\mu)^2+4(k-\mu)}=(v-1-2k)^2.	
\end{equation}
After some tedious computations, we see that when (\ref{eq size}) is satisfied, $v$ must be a square, and $\mu=(\frac{k}{\sqrt{v}-1})^2-\frac{k}{\sqrt{v}-1}$ or $\mu=(\frac{k}{\sqrt{v}+1})^2+\frac{k}{\sqrt{v}+1}$. Note that by Lemma \ref{rel2}, $\lambda=k-1+(1-\frac{v-1}{k})\mu$. Consequently if one of the duals of $D$ has the same size as that of $D$, then the strongly regular graph ${\rm Cay}(G,D)$ must be of Latin square or negative Latin square type.	
 
Conversely, when $(v,k,\lambda,\mu)=(\gu^2,a(\gu-1),\gu+a^2-3a,a^2-a)$, by Lemma \ref{relation}, $\alpha_1=\gu-a$ with multiplicity $f=a(\gu-1)$, and $\alpha_2=-a$ with multiplicity $g=(\gu+1-a)(\gu-1)$. Thus the dual $\{\psi\in \hat{G}:\psi(D)=\alpha_1\}$ has the same size as that of $D$. When $(v,k,\lambda,\mu)=(\gu^2,a(\gu+1),-\gu+a^2+3a,a^2+a)$, by Lemma \ref{relation}, $\alpha_1=a$ with multiplicity $f=(\gu-1-a)(\gu+1)$, and $\alpha_2=a-\gu$ with multiplicity $g=a(\gu+1)$. Thus the dual $\{\psi\in \hat{G}:\psi(D)=\alpha_2\}$ has the same size as that of $D$.

We will use Cayley graphs ${\rm Cay}(\F_q, D)$, where the connection sets are unions of cyclotomic classes, for the purpose of constructing $m$-ovoids in symplectic polar spaces. To this end, we define cyclotomic classes of finite fields. Let $q=p^{s}$ be a prime power, and let $\gamma$ be a fixed primitive element of $\F_q$. Let $N>1$ be a divisor of $q-1$. We define the $N^{\rm th}$ {\em cyclotomic classes} 
$C_i^{(N,q)}$ of $\F_q$ by
$$C_i^{(N,q)}=\{\gamma^{jN+i}\mid 0\leq j\leq \frac {q-1}{N}-1\},$$ where $0\leq i\leq N-1$. That is, $C_0^{(N,q)}$ is the subgroup of $\F_q^*$ consisting of all nonzero $N^{\rm th}$ powers in $\F_q$, and $C_i^{(N,q)}=\gamma^i C_0^{(N,q)}$, for $1\leq i\leq N-1$. In the sequel, if $N,q$ are clear from the context, we will simply write $C_i^{(N,q)}$ as $C_i$.

Suppose that $q=p^s$ with $p$ a prime, and let $e$ be any positive divisor of $s$. Let $\textup{Tr}_{q/p^e}:\mathbb{F}_{q}\rightarrow \mathbb{F}_{p^e}$ be the trace function from $\mathbb{F}_{q}$ to $\mathbb{F}_{\q}$, i.e., 
$$\tr_{q/p^e}(x)=x+x^{p^e}+\cdots+x^{p^{e(s/e-1)}}, \; \forall x\in\F_q.$$
Set $\omega_p:=\exp\left(\frac{2\pi\sqrt{-1}}{p}\right)$. Define $\psi_{\F_q}$: $\F_q\to \C^\ast$ by 
\[
\psi_{\F_q}(x)=\omega_p^{\tr_{q/p}(x)}, \; \forall x\in\F_q.
\]
The map $\psi_{\F_q}$ is a character of the additive group of $\F_q$, and it is called the {\it canonical} additive character of $\F_q$. For any $y\in \F_q$, we define $\psi_{\F_q,y}$: $\F_q\to \C^\ast$ by
\[
\psi_{\F_q,y}(x)=\psi_{\F_q}(xy), \; \forall x\in\F_q.
\]
It is well known that $\{\psi_{\F_q,y}\mid y\in \F_q\}=\widehat{(\F_q,+)}$.

Next we recall the following construction of strongly regular Cayley graphs given in \cite{Brouwer1999Journal}, see also \citep[Section 9.8.5]{Brouwer2012Spectra}. Suppose that $q=p^s$ with $p$ prime, and let $N$ be a proper divisor of $q-1$ such that  $p^{\ell}\equiv -1\pmod{N}$ for some positive integer $\ell$. Choose $\ell$ minimal and write $s=2\ell t$. Take a proper subset $J\subset \Z_N$ of size $u$. If $q$ is even, then the choice of $J$ is arbitrary; if $q$ is odd, then we require that $N|\frac{q-1}{2}$ and $J+\frac{q-1}{2}=J$. Set $D_J=\cup_{i\in J}C_i$. Then the graph $\Cay(\F_q,\,D_J)$
is strongly regular with eigenvalues
\begin{equation}\label{para in th}
\begin{split}
	k=\frac{q-1}{N}u ,\quad&\text{with multiplicity 1,} \\
	\alpha_1=\frac{u}{N}(-1+(-1)^t\sqrt{q}),\quad&\text{with multiplicitiy $q-1-k$,}\\
	\alpha_2=\frac{u}{N}(-1+(-1)^t\sqrt{q})+(-1)^{t+1}\sqrt{q},\quad& \text{with multiplicity $k$.}
\end{split}	
\end{equation}
To be specific, for $i=0,1,\ldots,N-1$,  we have
\begin{equation}\label{srg dual}
\psi_{\F_{q}}(\gamma^i D)=
\begin{cases}
\alpha_2,& \text{ if $\varepsilon^s=1$ and $i\in -J\, \pmod{N}$} \\
& \text{ \, \, or $\varepsilon^s=-1$ and $i\in -J+N/2\, \pmod{N}$}, \\
\alpha_1, & \text{ otherwise, }
\end{cases}
\end{equation}
where $\varepsilon=\begin{cases}
-1,   &\textup{ if $N$ is even and $\frac{p^{\ell}+1}{N}$ is odd;}\\
1,&\textup{otherwise.}	\end{cases}$ 

\noindent The graph $\Cay(\F_q,\,D_J)$ is of Latin square type (resp. negative Latin square type) if $t$ is odd (resp. even).

\section{Construcions of $m$-ovoids in $\W(2r-1,p^e)$ via strongly regular Cayley graphs}

Throughout the rest of this paper, we fix the following notation. Let $q=p^s$, where $p$ is an odd prime and $s=2er$ for some positive integers $e$ and $r\geq 2$. We view $\F_q$ as a $2r$ -dimensional vector space over $\F_{\q}$ (a subfield of $\F_q$), and denote this $\F_{\q}$-vector space by $V$. As usual, for a nonzero $v\in V$, we write $\la v\ra$ for the projective point in the projective space $\PG(V)$ corresponding to the 1-dimensional $\F_{\q}$-subspace spanned by $v$. We will equip a bilinear form on $V$ as follows. Let $L(X)= \sum_{i=0}^{2r-1}c_iX^{p^{ie}}\in \F_q[X]$ be a linearized polynomial. Define $f: V\times V \rightarrow \F_{\q}$ by 
$$f(x,y)=\tr_{q/p^e}(xL(y)), \forall (x,y)\in V\times V.$$
Then $f$ is an $\F_{\q}$-bilinear form on $V$.

\begin{lemma}\label{lem_AL}
With notation as above, the $\F_{\q}$-bilinear form  $f$ on $V$ is alternating if and only $c_0=0$ and $c_{2r-i}^{p^{ie}}=-c_i$ for $1\le i\le 2r-1$. Moreover, the form $f$ is non-degenerate if and only if $x\mapsto L(x)$ is a bijection from $\F_q$ to itself.
\end{lemma}
\begin{proof} The form $f$ is alternating if and only if $f(x,x)=0$ for all $x\in V$. We have
\begin{align*}
f(x,x)&=\sum_{i=0}^{2r-1}\tr_{q/p^e}\left(c_ix^{1+p^{ie}}\right)=\sum_{i,j=0}^{2r-1}c_i^{p^{je}}x^{p^{je}+p^{(i+j)e}}\\
&=\sum_{0\le i<k\le 2r-1} (c_{k-i}^{p^{ie}}+c_{2r-k+i}^{p^{ke}})x^{p^{ie}+p^{ke}}+\sum_{j=0}^{2r-1}c_0^{p^{je}}x^{2p^{je}}.
\end{align*}
We view $f(x,x)$ as a polynomial in $x$ with coefficients in $\F_q$. As can be seen from the above expression, $f(x,x)$ has degree less than or equal to $q-1$. The polynomial $f(x,x)$ vanishes at every element of $\F_q$ if and only if it is the zero polynomial. The first claim of the lemma now follows by comparing the coefficients.

Assume that $f$ is non-degenerate.  Suppose that $f(x,y)=\tr_{q/p^e}(xL(y))=0$ for all $x\in\F_q$. We must have $L(y)=0$ since $\tr_{q/p^e}$ is a nontrivial linear form on $V$. By the assumption that $f$ is non-degenerate, we have $y=0$. It follows that $x\mapsto L(x)$ is a bijection since $L$ is linearized. The proof of the converse is straightforward. This proves the second claim of the lemma.
\end{proof}
 
Let $L(X)= \sum_{i=1}^{2r-1}c_iX^{p^{ie}}\in \F_q[X]$ be a linearized permutation polynomial such that $f(x,y)=\tr_{q/p^e}(xL(y))$ is a non-degenerate alternating form. Equip $V=(\F_q,\,+)$ with the alternating form $f$, and this will be our model for the symplectic polar space $\W(2r-1,p^e)$ of rank $r$. For any nonzero $y\in V$, we define
$$\langle y\rangle^{\perp}=\{\langle x\rangle \mid f(x,y)=0, \; x\in V\setminus \{0\}\}.$$
Also for any $y\in V$ we define $\Psi_y\in\widehat{(\F_q,+)}$ as follows:
\[
\Psi_y(x)=\psi_{\F_q} (xL(y))=\psi_{\F_{p^e}}(f(x,y)).
\]
It is well known that $\{\Psi_y \mid y\in V\}=\widehat{(\F_q,+)}$. We now give the definition of self-dual partial difference sets in $(\F_q,+)$. (Note that such a self-dual partial difference set is necessarily of Latin square type or negative Latin square type by the discussions in the last section.) Let $D$ be an $\F_{p^e}^*$-invariant subset of $\F_q^*$. That is, $D$ is a union of some cosets of $\F_{p^e}^*$ in $\F_q^*$. Assume that $\Cay(\mathbb{F}_q,D)$ is a strongly regular graph of negative Latin square type, and with parameters
\[
(q,|D|,-\sqrt{q}+(\frac{|D|}{\sqrt{q}+1})^2+\frac{3|D|}{\sqrt{q}+1},(\frac{|D|}{\sqrt{q}+1})^2+\frac{|D|}{\sqrt{q}+1}). \]
Let $D^*\subset\F_q\setminus\{0\}$ be such that $\{\Psi_{y}\in \widehat{(\F_q,+)}\mid y\in D^*\}$ is one of the Delsarte duals. We say that $D$ is {\it self-dual} if $D^*=D$. Now we state the connection between $m$-ovoids in $\W(2r-1,p^e)$ and self-dual partial difference sets in $(\F_q,+)$ explicitly.


\begin{thm}\label{th main}
With notation as above, let $D$ be an $\F_{p^e}^*$-invariant subset of $\F_q^*$ such that $|D| \equiv 0\pmod{(\sqrt{q}+1)(p^e-1)}$. Then the set $\mathcal{M}=\{\langle v\rangle:\,v\in D\}$ is a $\frac{|D|}{(p^e-1)(\sqrt{q}+1)}$-ovoid in $\W(2r-1,p^e)$ if and only if $D$ is a self-dual partial difference set with parameters $(q,|D|,-\sqrt{q}+(\frac{|D|}{\sqrt{q}+1})^2+\frac{3|D|}{\sqrt{q}+1},(\frac{|D|}{\sqrt{q}+1})^2+\frac{|D|}{\sqrt{q}+1})$.
\end{thm}

\begin{proof} Assume that $D$ is a self-dual partial difference set with the above parameters. Then by Lemma \ref{relation}, the graph $\Cay(\F_q,\,D)$ has two eigenvalues $\alpha_1=\frac{|D|}{\sqrt{q}+1}$ (with multiplicity $m_1=q-1-|D|$), and $\alpha_2=\frac{|D|}{\sqrt{q}+1}-\sqrt{q}$ (with multiplicity $m_2=|D|$); and furthermore since $D$ is self-dual, we have $\Psi_{y}(D)=\alpha_2$ if $y\in D^*=D$, and $\Psi_{y}(D)=\alpha_1$ if $y\in\F_q^*\setminus D$.

Write $\cM:=\{\la v_i\ra:\,1\le i\le M\}$, where $M=|D|/(p^e-1)$. Then $D=\{\theta v_i:\,1\le i\le M,\,\theta\in\F_{\q}^*\}$. For any nonzero $y\in V$, we have
\begin{align*}
\Psi_y(D)&=\sum_{i=1}^{M}\sum_{\theta\in \mathbb{F}_{p^e}^{\ast}}\Psi_y(\theta v_i)\\
&=\sum_{i=1}^{M}\sum_{\theta\in \mathbb{F}_{p^e}^{\ast}}\psi_{\mathbb{F}_{p^e}}(\theta f(v_i,y))\\
&=-M+\sum_{i=1}^{M}\sum_{\theta\in \mathbb{F}_{p^e}}\psi_{\mathbb{F}_{p^e}}(\theta f(v_i,y))\\
&=-M+p^e\cdot|\{1\le i\le M:\,f(v_i,y)=0\}|\\
&=p^e\cdot|\langle y\rangle^{\perp}\cap \mathcal{M}|-|\mathcal{M}|.
\end{align*}

Now from $\Psi_{y}(D)=\alpha_2$ if $y\in D$, we obtain $$|\langle y\rangle^{\perp}\cap \mathcal{M}|=\frac{|D|}{(p^e-1)(\sqrt{q}+1)}\cdot (p^{e(r-1)}+1) -p^{e(r-1)}\; {\rm if}\; \la y\ra\in {\mathcal M};$$ and from $\Psi_{y}(D)=\alpha_1$ if $y\in\F_q^*\setminus D$, we obtain $$|\langle y\rangle^{\perp}\cap \mathcal{M}|=\frac{|D|}{(p^e-1)(\sqrt{q}+1)}\cdot (p^{e(r-1)}+1)\; {\rm if} \; \la y\ra \not\in {\mathcal M}.$$ Therefore ${\mathcal M}$ is a $\frac{|D|}{(p^e-1)(\sqrt{q}+1)}$-ovoid in $\W(2r-1,p^e)$ by Lemma~\ref{equivdef}. The converse can be proved by simply running the above reasoning backwards. The proof is complete.
\end{proof}

We make some comments on self-duality. Let $D$ be an $\F_{p^e}^*$-invariant subset of $\F_q^*$. Assume that $\Cay(\mathbb{F}_q,D)$ is a strongly regular graph of negative Latin square type with parameters
\[
(q,|D|,-\sqrt{q}+(\frac{|D|}{\sqrt{q}+1})^2+\frac{3|D|}{\sqrt{q}+1},(\frac{|D|}{\sqrt{q}+1})^2+\frac{|D|}{\sqrt{q}+1}). \]
Define 
\begin{equation}\label{shifteddual}
D'=\{y\in \F_q\mid  \psi_{\F_q,y}(D)=\alpha_2\}.
\end{equation}
Noting that $\Psi_y(D)=\psi_{\F_q,L(y)}(D)$, we see that $D$ is self-dual if and only if $L(D)=D'$.

We will use Theorem~\ref{th main} for the purpose of constructing $m$-ovoids in $\W(2r-1,p^e)$. The first step is to equip $V$ with a concrete non-degenerate alternating form by choosing $L(X)$ carefully. Let $\gamma$ be a primitive element of $\F_q$ (recall that $q=p^{s}$ is an odd prime power, and $s=2er$), and set $\delta=\gamma^{\frac{\sqrt{q}+1}{2}}\in \mathbb{F}_q$. Then we have $\delta^{\sqrt{q}}=-\delta$.  Let $L(X)=\delta X^{\sqrt{q}}$. By Lemma \ref{lem_AL}, $f(x,y):=\tr_{q/p^e}(xL(y))$ is a non-degenerate alternating form defined on $V$. We will take $V$ equipped with this $f$ as our model for the symplectic polar space $\W(2r-1,p^e)$ in the rest of this paper.

We further assume that $q=p^{s}$ with $s=2\ell t$ and $t$ even. Take $N=p^{\ell}+1$, and let $C_0,\ldots,C_{N-1}$ be the $N^{\rm th}$ cyclotomic classes of $\F_q$. In this case, it is easy to verify that $N|\frac{q-1}{2}$. Let $J$ be a proper subset of $\Z_N$ of size $u$, and set $D_J=\cup_{i\in J}C_i$. By \cite{Brouwer1999Journal} (see the discussions in the end of Section 2), the graph $\Cay(\F_q,D_J)$ is a strongly regular graph with negative Latin square type parameters
\[
(q,\frac{u(\sqrt{q}-1)}{N}(\sqrt{q}+1),-\sqrt{q}+(\frac{u(\sqrt{q}-1)}{N})^2+\frac{3u(\sqrt{q}-1)}{N}, (\frac{u(\sqrt{q}-1)}{N})^2+\frac{u(\sqrt{q}-1)}{N}).
\]
Moreover, its eigenvalues are
\begin{equation}\label{psi}
\psi_{\mathbb{F}_q}(\gamma^iD_J)=\begin{cases}\frac{u}{N}(\sqrt{q}-1)-\sqrt{q},\,\,\, &\text{if } -i \;(\text{mod }N) \in J,\\
\frac{u}{N}(\sqrt{q}-1),\,\,\, &\text{otherwise}.	
\end{cases}
\end{equation}
It follows that $D'_J=\cup_{-i\in J}C_i$, where $D'_J$ is defined in (\ref{shifteddual}). In order to use Theorem \ref{th main} to obtain $m$-ovoids in $\W(2r-1,p^e)$, we need $D_J$ to be self-dual. To avoid possible overlap with the $m$-ovoids obtained by the field reduction method, we further assume that $r$ is odd.

\begin{lemma}\label{lem_sig_inv}
With notation as above, $D_J$ is self-dual if and only if $J$ is $\sigma$-invariant, where $\sigma:\, i\mapsto -1-i\pmod{N}$.
\end{lemma}
\begin{proof}
We have $L(C_i)=\gamma^{\frac{\sqrt{q}+1}{2}+\sqrt{q}i}C_0$, i.e., $L$ maps $C_i$ to $C_{\tau(i)}$, where $\tau(i):=\frac{\sqrt{q}+1}{2}+\sqrt{q}i\pmod{N}$. Note that $\sqrt{q}\equiv 1\pmod{N}$, and $\frac{\sqrt{q}-1}{2}=\frac{(p^{\ell t}-1)}{(p^{2\ell} -1)}\cdot\frac{(p^{\ell}-1)}{2}\cdot N\equiv 0\pmod{N}$ since $t$ is even by assumption. Therefore $\tau(i)=i+1\pmod{N}$. The partial difference set $D_J$ is self-dual if and only if $L(D_J)=D'_J$ which in turn is equivalent to $\{i+1\pmod{N}:\,i\in J\}=\{-i\pmod{N}:\,i\in J\}$, i.e., $-J-1=J$. The proof of the lemma is complete.
\end{proof}

\begin{lemma}\label{lem_Fpr_inv}
With notation as above, $D_J$ is $\F_{p^e}^*$-invariant if and  only if $J$ is invariant under the map $\rho:\,i\mapsto i+2d_0\pmod{N}$, where $d_0$ is an odd integer defined by
\begin{equation}\label{gcd}
	d_0:=\gcd\left(\frac{N}{2},\, \frac{\sqrt{q}-1}{p^e-1}\right).
\end{equation}
\end{lemma}
\begin{proof}
Note that $\F_{p^e}^*=\la\gamma^{(q-1)/(p^e-1)}\ra$. So $D_J$ is $\F_{p^e}^*$-invariant if and  only if $J$ is invariant under the map $i\mapsto i+\frac{q-1}{p^e-1}\pmod{N}$. We have $\sqrt{q}+1\equiv2\pmod{N}$ by the assumption that $t$ is even. It follows that $\gcd(N,\frac{q-1}{p^e-1})=\gcd(N, \frac{2(\sqrt{q}-1)}{p^e-1})=2d_0$ with $d_0$ as defined in \eqref{gcd}.The conclusion of the lemma follows.
\end{proof}

\begin{thm}\label{thm_main_const}
Let $p$ be an odd prime, and $q=p^s$ with $s=2er=2\ell t$ for some positive integers $e,\,r,\,\ell,\,t$, where $r$ is odd and $t$ is even. Set $N=p^{\ell}+1$, and let $d_0$ be defined as in \eqref{gcd}. If $d_0>1$, then there exists a $\frac{b(\sqrt{q}-1)}{d_0(p^e-1)}$-ovoid in $\W(2r-1,p^e)$ for each integer $1\le b\le d_0-1$.
\end{thm}
\begin{proof}
We continue with the above notation. By Lemma \ref{lem_sig_inv} and Lemma \ref{lem_Fpr_inv}, the set $D_J$ is $\F_{p^e}^*$-invariant and satisfies $L(D_J)=D'_J$ if and only if $J$ is invariant under the map  $\rho:\,i\mapsto i+2d_0\pmod{N}$ and $\sigma:\,i\mapsto -1-i\pmod{N}$. We claim that $\langle \rho,\sigma\rangle=\{1,\rho,\rho^2,...,\rho^{\frac{N}{2d_0}-1},\sigma,\rho\sigma, \rho^2\sigma,...,\rho^{\frac{N}{2d_0}-1}\sigma\}$ by the fact that $\langle \rho,\sigma\rangle $ is a dihedral group $D_{\frac{N}{d_0}}$, cf. \citep[2.24]{Rose1978A}. Thus each  $\la\rho,\,\sigma\ra$-orbit ${\mathcal O}$ on $\Z_N$ has equal length $\frac{N}{d_0}$, and the corresponding union $D_{\mathcal O}=\cup_{i\in {\mathcal O}}C_i$ of cyclotomic classes has size $|C_0|\cdot \frac{N}{d_0}=\frac{q-1}{d_0}$. This number is divisible by $(\sqrt{q}+1)(p^e-1)$, since $d_0$ divides $\frac{\sqrt{q}-1}{p^e-1}$ by \eqref{gcd}. Therefore, $D_{\mathcal O}$ is an $m$-ovoid in $\W(2r-1,p^e)$ with $m=\frac{\sqrt{q}-1}{d_0(p^e-1)}$. By taking union of $b$ such $D_{\mathcal O}$'s, we get $bm$-ovoids in $\W(2r-1,p^e)$ for $1\le b\le d_0-1$.
\end{proof}

For the following discussions, we choose $r$ to be an odd prime $p_0$, and give explicit conditions that guarantee $d_0>1$, where $d_0$ is defined in \eqref{gcd}.  This excludes the possibility that the resulting $m$-ovoids from Theorem \ref{thm_main_const} come from field reduction. Recall that $s=2e p_0=2\ell t$ with $t$ even. We consider two cases.

\begin{enumerate}
\item[(A)] First consider the case $p_0\mid t$. Write $t=p_0t_0$. So $t_0$ is even and $e=\ell t_0$, and
    \[
       \frac{\sqrt{q}-1}{p^e-1}=\sum_{i=0}^{p_0-1}p^{ie}=\sum_{i=0}^{p_0-1}p^{i\ell t_0}\equiv p_0 \pmod{p^{\ell}+1}.
    \]
    Therefore, $d_0=\gcd(N/2,p_0)=\gcd(N,p_0)$. In this case, $d_0>1$ if and only if $p_0\mid (p^{\ell}+1)$, which in turn is equivalent to $d_0=p_0$. In Table \ref{table_1}, we give some examples of $m$-ovoids constructed by using Theorem \ref{thm_main_const} in this case.

\item[(B)] Next consider the case where $p_0$ does not divide $t$.  In this case, from $p_0e=\ell t$, we deduce that $p_0\mid \ell$. Write $\ell=\ell_0p_0$. Then $e=\ell_0t$. We show that $d_0=\gcd(N/2,\frac{\sqrt{q}-1}{p^e-1})$ is always greater than $1$. We have $\sqrt{q}-1=p^{\ell t}-1\equiv (-1)^t-1=0\pmod{N}$ by the fact $t$ is even, so $N/2$ divides $\sqrt{q}-1$. On the other hand, $\gcd(p^e-1,N/2)$ divides $\gcd(p^e-1,p^{2\ell}-1)=p^{\gcd(e,2\ell)}-1=p^{2\ell_0}-1$ by the fact $p_0\nmid t$. Since $p^{2\ell_0}-1<N/2=(p^{\ell_0p_0}+1)/2$, we see that $N/2$ does not divide $p^e-1$. We conclude that $d_0>1$. In Table \ref{table_2}, we give some examples of $m$-ovoids constructed by using Theorem \ref{thm_main_const} in this case.
\end{enumerate}


\begin{table}[!htbp]
\caption{$m$-ovoids constructed from Theorem \ref{thm_main_const} in the case $r=p_0$ odd prime, $p_0\mid t$}\label{table_1}
\begin{center}
\scalebox{0.82}{
\begin{tabular}{|c|c|c|c|c|c|c|c p{7cm}|}
\hline
$p_0$ & $p$ &  $\ell$&$t$&$d_0$& $W(2p_0-1,p^e)$ &   $m$ \\
\hline
 3&$p$ odd &1& 6$k$, $k\in \mathbb{Z}^+$&3&$W(5,p^{2k})$&$\frac{b}{3}(p^{4k}+p^{2k}+1)$, $b\in\{1,2\}$ \\
\hline
5 & 3 &2  & 10$k$, $k\in \mathbb{Z}^+$& 5 &$W(9,3^{4k})$&$\frac{b(3^{20k}-1)}{5(3^{4k}-1)}$, $1\le b\le 4$\\
\hline
5&7&2&10$k$, $k\in \mathbb{Z}^+$&5&$W(9,7^{4k})$&$\frac{b(7^{20k-1})}{5(7^{4k}-1)}$, $1\le b\le 4$\\
\hline
5&13&2&10$k$, $k\in \mathbb{Z}^+$&5&$W(9,13^{4k})$&$\frac{b(13^{20k}-1)}{5(13^{4k}-1)}$, $1\le b\le 4$\\
\hline
5&17&2& 10$k$, $k\in \mathbb{Z}^+$& 5&$W(9,17^{4k})$&$\frac{b(17^{20k}-1)}{5(17^{4k}-1)}$, $1\le b\le 4$\\
\hline
5&19&1& 10$k$, $k\in \mathbb{Z}^+$& 5&$W(9,19^{2k})$&$\frac{b(19^{10k}-1)}{5(19^{2k}-1)}$, $1\le b\le 4$\\
\hline
7&3&3& 14$k$, $k\in \mathbb{Z}^+$& 7&$W(13,3^{6k})$&$\frac{b(3^{42k}-1)}{7(3^{6k}-1)}$, $1\le b\le 6$\\
\hline
7&5&3& 14$k$, $k\in \mathbb{Z}^+$&7& $W(13,5^{6k})$&$\frac{b(5^{42k}-1)}{7(5^{6k}-1)}$, $1\le b\le 6$\\
\hline
7&13&3& 14$k$, $k\in \mathbb{Z}^+$&7& $W(13,13^{6k})$&$\frac{b(13^{42k}-1)}{7(13^{6k}-1)}$, $1\le b\le 6$\\
\hline
11&7&5& 22$k$, $k\in \mathbb{Z}^+$& 11&$W(21,7^{10k})$&$\frac{b(7^{110k}-1)}{11(7^{10k}-1)}$, $1\le b\le 10$\\
\hline
11&13&5& 22$k$, $k\in \mathbb{Z}^+$& 11&$W(21,13^{10k})$&$\frac{b(13^{110k}-1)}{11(13^{10k}-1)}$, $1\le b\le 10$\\
\hline
11&17&5& 22$k$, $k\in \mathbb{Z}^+$& 11&$W(21,17^{10k})$&$\frac{b(17^{110k}-1)}{11(17^{10k}-1)}$, $1\le b\le 10$\\
\hline
11&19&5& 22$k$, $k\in \mathbb{Z}^+$& 11&$W(21,19^{10k})$&$\frac{b(19^{110k}-1)}{11(19^{10k}-1)}$, $1\le b\le10$\\
\hline
13&5&2& 26$k$, $k\in \mathbb{Z}^+$&13& $W(25,5^{4k})$&$\frac{b(5^{52k}-1)}{13(5^{4k}-1)}$, $1\le b\le 12$\\
\hline
13&7&6& 26$k$, $k\in \mathbb{Z}^+$& 13&$W(25,7^{12k})$&$\frac{b(5^{156k}-1)}{13(5^{12k}-1)}$, $1\le b\le 12$\\
\hline
\end{tabular}}
\end{center}
\end{table}

\begin{table}[!htbp]
\caption{{$m$-ovoids constructed from Theorem \ref{thm_main_const} in the case $r=p_0$ odd prime, $p_0\nmid t$}}\label{table_2}
\begin{center}
\scalebox{0.9}{
\begin{tabular}{|c|c|c|c|c|c|c| p{7cm}|}
\hline
$p_0$ & $p$ &  $\ell$&$t$&$d_0$& $W(2p_0-1,p^e)$ &    $m$ \\
\hline
3&3&3&2$k$, $3\nmid k$&7&$W(5,3^{2k})$&$\frac{b(3^{6k}-1)}{7(3^{2k}-1)}$, $1\le b\le 6$\\
\hline
3&5&3&$2k$, $3\nmid k$&21&$W(5,5^{2k})$&$\frac{b(5^{6k}-1)}{21(5^{2k}-1)}$, $1\le b\le 20$\\
\hline
3&7&3&$2k$, $3\nmid k$&43&$W(5,7^{2k})$&$\frac{b(7^{6k}-1)}{43(7^{2k}-1)}$, $1\le b\le42$ \\
\hline
5&3&5&$2k$, $5\nmid k$&61&$W(9,3^{2k})$&$\frac{b(3^{10k}-1)}{61(3^{2k}-1)}$, $1\le b\le60$\\
\hline
5&5&5&$2k$, $5\nmid k$&521&$W(9,5^{2k})$&$\frac{b(5^{10k}-1)}{521(5^{2k}-1)}$, $1\le b\le520$\\
\hline
7&3&7&$2k$, $7\nmid k$&547&$W(13,3^{2k})$&$\frac{b(3^{14k}-1)}{547(3^{2k}-1)}$, $1\le b\le546$\\
\hline
7&5&7&$2k$, $7\nmid k$&13021&$W(13,5^{2k})$&$\frac{b(5^{14k}-1)}{13021(5^{2k}-1)}$, $1\le b\le13020$\\
\hline
11&3&11&$2k$, $11\nmid k$&44287&$W(21,3^{2k})$&$\frac{b(3^{22k}-1)}{44287(3^{2k}-1)}$, $1\le b\le44286$\\
\hline
\end{tabular}}
\end{center}
\end{table}

It was conjectured in \cite{Bamberg2017} that if an $m$-ovoid exists in $\W(2r-1,p^e)$ with $r>2$, then $m\geq cp^{e(r-2)}$ for some positive constant $c>0$. From Theorem 3.5, we know that there exist $m$-ovoids in $\W(2r-1,p^e)$ with $m=\frac{b(\sqrt{q}-1)}{d_0(p^e-1)}$ if $d_0>1$. We compute

\begin{equation*}
\frac{m}{p^{e(r-2)}}=	\frac{b(\sqrt{q}-1)}{d_0(p^e-1)p^{e(r-2)}}=\frac{b(p^{er}-1)}{d_0(p^{e(r-1)}-p^{e(r-2)})}
\end{equation*}

In the following we show that the aforementioned conjecture is false.

\begin{example}
	In the case (B), when $p=3,r=p_0=5,t=2$, we have $d_0=\gcd{(\frac{3^{5\ell_0}+1}{2},\frac{3^{10\ell_0}-1}{3^{2\ell_0}-1})}$. It follows that $\frac{3^{5\ell_0}+1}{3^{\ell_0}+1}\mid d_0$; thus $\frac{m}{p^{e(p_0-2)}}\leq\frac{b(3^{4\ell_0}+3^{3\ell_0}+3^{2\ell_0}+3^{\ell_0}+1)}{3^{6\ell_0}}$, for some fixed $b$, and $\lim_{\ell_0\rightarrow \infty}\frac{m}{p^{e(p_0-2)}}=0$.
\end{example}
\begin{example}
In the case (B), when $p=5,r=p_0=7,t=2$, we have $d_0=\gcd{(\frac{5^{7\ell_0}+1}{2},\frac{5^{14\ell_0}-1}{5^{2\ell_0}-1})}$. It follows that $\frac{5^{7\ell_0}+1}{5^{\ell_0}+1}\mid d_0$; thus $\frac{m}{p^{e(p_0-2)}}\leq \frac{b(5^{7\ell_0}-1)}{(5^{\ell_0}-1)5^{10\ell_0}}$, for some fixed $b$, and $\lim_{\ell_0\rightarrow \infty}\frac{m}{p^{e(p_0-2)}}=0$. 	
\end{example}
\begin{example}
In the case (B), when $p=5,r=p_0=11,t=2$, we have $d_0=\gcd{(\frac{5^{11\ell_0}+1}{2},\frac{5^{22\ell_0}-1}{5^{2\ell_0}-1})}$. It follows that $\frac{5^{11\ell_0}+1}{5^{\ell_0}+1}\mid d_0$; thus $\frac{m}{p^{e(p_0-2)}}\leq \frac{b(5^{11\ell_0}-1)}{(5^{\ell_0}-1)5^{18\ell_0}}$, for some fixed $b$, and $\lim_{\ell_0\rightarrow \infty}\frac{m}{p^{e(p_0-2)}}=0$. 	
\end{example}

\section{Conclusion}
In this paper, we develop a new method for constructing $m$-ovoids in finite symplectic spaces. We use some ``special" strongly regular Cayley graphs ${\rm Cay}(\F_{q}, D)$ from uniform cyclotomy in \cite{Brouwer1999Journal} and equip the ambient finite field $\F_{q}$ with a non-degenerate alternating form $f$ so that the connection set $D$ gives rise to an $m$-ovoid in the symplectic space $(\F_q, f)$. In this way we are able to obtain $m$-ovoids in high rank symplectic spaces which do not come from field reduction. We remark that there have been extensive investigations on constructions of strongly regular Cayley graphs from cyclotomy in recent years, cf. \cite{Feng2012Strongly, Momihara2014Certain,Feng2015Con,Momihara2018Strongly}, and it will be of interest to examine whether further new $m$-ovoids can arise from Theorem \ref{th main}.

We mention in passing that we only considered constructing $m$-ovoids in symplectic spaces from strongly regular graphs of negative Latin square type, but did not discuss similar constructions of $i$-tight sets in symplectic spaces from strongly regular graphs of Latin square type. The reason is as follows:  each component of a spread of $\W(2r-1,p^e)$ is a 1-tight set, thus there exist $i$-tight sets for all $i$ in $\W(2r-1,p^e)$. 

The results in this paper show that there are many more $m$-ovoids in finite symplectic spaces than previously thought. Still it remains an interesting problem to determine for which values of $m$ there exists an $m$-ovoid in $\W(2r-1,p^e)$.

\vspace{0.1in}

\noindent{\bf Acknowledgements.} The research of Tao Feng is supported the National Natural Science Foundation of China grant 11771392, and the research of Qing Xiang is supported by an NSF grant DMS 1855723.

\scriptsize
\setlength{\bibsep}{0.5ex}  


\end{document}